\documentclass[12pt]{amsart}
\usepackage{amsfonts,newlfont,latexsym,enumerate}
\usepackage{euscript}
\usepackage{amssymb,amscd}
\usepackage{color}

\newtheorem{theorem}{Theorem}[section]

     \newtheorem{cor}[theorem]{Corollary}
     \newtheorem{prop}[theorem]{Proposition}
\theoremstyle{definition}
     \newtheorem{remark}[theorem]{Remark}
     \newtheorem{example}[theorem]{Example}
     
     \newtheorem{question}[theorem]{Question}

\newcommand{\inv}{^{-1}}
\newcommand{\N}{\mathbf N}
\newcommand{\F}{\mathbf F}
\newcommand{\Z}{\mathbf Z}
\newcommand{\A}{\mathbf A}
\newcommand{\K}{\mathbf K}

\newcommand{\DL}{\textnormal{DL}}
\newcommand{\Ker}{\textnormal{Ker}}
\newcommand{\Aut}{\textnormal{Aut}}
\newcommand{\HNN}{\textnormal{HNN}}

\DeclareMathOperator{\Isom}{Isom}

\begin{document}

\title{Cross-wired lamplighter groups}

\thanks{D.\ F.\ was partially supported by NSF grant DMS 0643546}


\author[Cornulier]{Yves de Cornulier}
\address[Y.C.]{Laboratoire de Math\'ematiques\\
B\^atiment 425, Universit\'e Paris-Sud 11\\
91405 Orsay\\FRANCE}
\email{yves.cornulier@math.u-psud.fr}

\author[Fisher]{David Fisher}
\author[Kashyap]{Neeraj Kashyap}
\address[D.F. and N.K.]{Department of Mathematics
Indiana University -- Bloomington
Rawles Hall\\
Bloomington\\ IN 47401\\ USA
}
\email{fisherdm@indiana.edu, nkashyap@indiana.edu}

\date{June 2, 2012}
\subjclass[2010]{05C63 (primary); 20E22 (secondary)}


\maketitle
\begin{abstract}We give a necessary and sufficient condition for a locally compact group
to be isomorphic to a closed cocompact subgroup in the isometry group of a Diestel-Leader graph. As a consequence
of this condition, we see that every cocompact lattice in the isometry group of a Diestel-Leader
graph admits a transitive, proper action on some other Diestel-Leader graph. We also
give some examples of lattices that are not virtually lamplighters. This implies the class of discrete groups commensurable to lamplighter groups is not
closed under quasi-isometries and, combined with work of
Eskin, Fisher and Whyte, gives a characterization of
their quasi-isometry class.
\end{abstract}



\section{Introduction}

A lamplighter group is a group of the form $F\wr \mathbf{Z}$ where
$F$ is some finite group. For one particular choice of generators, the Cayley graphs of these groups are examples of Diestel-Leader graphs $\DL(n,n)$ (\cite{DiestelLeader,Woess, Wortman}), which can be defined as follows (see Section \ref{idl} for details): let $T_n$ be the $(n+1)$-regular tree and $b$ a Busemann function on $T_n$. The Diestel-Leader graph of degree $n$ is defined as
$$\DL(n,n)=\{(x,y)\in T_n\times T_n:b(x)+b(y)=0\}.$$
This is a connected graph, where $\{(x,y),(x',y')\}$ is an edge whenever both $\{x,x'\}$ and $\{y,y'\}$ are edges.
Eskin, Fisher, and Whyte have shown (\cite{EFW1, EFW2, EFW3}) that a finitely generated group is quasi-isometric to a
lamplighter group $F\wr\Z$ if and only if it acts properly and cocompactly by
isometries on a Diestel-Leader graph $\DL(n,n)$, where $n$ and $|F|$ have a common power, i.e.\ it is a cocompact lattice in the isometry group
of some Diestel-Leader graph. In this paper, we address the
question of whether or not the lamplighter groups are
quasi-isometrically rigid by studying the cocompact lattices in the isometry group $\Isom(\DL(n,n))$
of the Diestel-Leader graphs. While we show that these lattices are not necessarily lamplighter groups,
we do give an algebraic characterization of them. We give cocompact lattices in $\Isom(\DL(n,n))$ the name
of {\em cross-wired lamplighters}, a terminology we explain briefly at the end of the paper.

An isometry of $\DL(n,n)$ is called {\it positive} if it is the restriction to $\DL(n,n)$ of some product isometry $(\alpha,\beta)\in\Aut(T_n\times T_n)$ of $T_n\times T_n$. It can be shown that a non-positive isometry of $\DL(n,n)$ is the composition of a positive isometry and the flip $(x,y)\mapsto (y,x)$, so that the group $\Isom^+(\DL(n,n))$ of positive isometries has index two in $\Isom(\DL(n,n))$ (see Section \ref{idl}).  Rather than just proving a theorem about lattices in $\Isom(\DL(n,n))$, we will prove a more general result about closed, cocompact subgroups of $\Isom(\DL(m,n))$.  In the case when $m \neq n$, $\Isom(\DL(m,n))$ is not unimodular, so there are no lattices, but understanding closed, cocompact subgroups is also of interest.  When $m \neq n$, all isometries are positive.

\begin{theorem}Suppose $m,n\ge 2$.

\begin{enumerate}[(a)]
\item\label{tta} Let $\Gamma$ be a closed, cocompact subgroup of $\Isom^+(\DL(m,n))$. Then $\Gamma$ has a unique open normal subgroup $H$ such that $\Gamma/H$ is infinite cyclic. Moreover, if $t$ is any element of $\Gamma$ mapping to a generator of $\Gamma/H$, then $H$ has two open subgroups $L,L'$ such that
\begin{itemize}
\item $tLt^{-1}$ and $t^{-1}L't$ are subgroups of index $m$ and $n$ in $L$ and $L'$ respectively;
\item $\bigcup_{k\in\Z}t^{-k}Lt^k=\bigcup_{k\in\Z}t^{k}L't^{-k}=H$ (these are increasing unions);
\item $L\cap L'$ is compact;
\item $LL'=H$ (that is, the double coset space $L\backslash H/L'$ is  reduced to a point).
\end{itemize}
Moreover, $\Gamma$ has no non-trivial compact normal subgroup and if $\Gamma$ is discrete, then $m=n$.

\item\label{ttb} Conversely, let $\Gamma$ be a locally compact group, with a semidirect product decomposition $\Gamma=H\rtimes\langle t\rangle$, with $H$ non-compact. Assume that $H$ has open subgroups $L,L'$ such that
\begin{itemize}
\item $tLt^{-1}$ and $t^{-1}L't$ are finite index subgroups, of index $m$ and $n$, in $L$ and $L'$ respectively;
\item $\bigcup_{k\in\Z}t^{-k}Lt^k=\bigcup_{k\in\Z}t^{k}L't^{-k}=H$ (these are increasing unions);
\item $L\cap L'$ is compact;
\item the double coset space $L\backslash H/L'$ is finite of cardinality $d$.
\end{itemize}
Then $H$ is locally finite, and $\Gamma$ has a proper, transitive action on $\DL(m,n)$, whose kernel is $\bigcap_{k\in\Z}t^{-k}(L\cap L')t^k$.  Moreover
if $\Gamma$ is discrete, then it is finitely generated and $m=n$.
\end{enumerate}\label{mainthm}
\end{theorem}




The combination of the two main points of the theorem yields a result we find surprising:
namely that any lattice $\Gamma<\Isom(\DL(n,n))$, or more generally any closed cocompact subgroup
 $\Gamma< \Isom(\DL(m,n))$ has an index two subgroup that admits a proper transitive action on
 another Diestel-Leader graph. We will also give a short geometric
proof of this fact that avoids the need to pass to a subgroup of index two, see Proposition \ref{proposition:transitive} below.  Of course,   any lattice $D$ in the isometry group of
any graph admits a proper transitive action on some graph, namely its own Cayley graph.
What is surprising here is that the transitive action is on another Diestel-Leader graph.
The only other class of graphs for which a similar statement is known are trees. As Diestel-Leader
graphs are clearly closely related to trees, it would be interesting to further study the analogies between cross-wired lamplighters and tree lattices.

We next provide examples of cross-wired lamplighters.

\begin{example}[lamplighters]
A standard wreath product $F\wr\Z$, where $F$ is any non-trivial finite group, is a cross-wired lamplighter (these examples were used by Erschler (Dyubina) \cite{Dyubina} to observe that being virtually solvable is not a quasi-isometry invariant). Here $L=F^{\Z_{\ge 0}}$ and $L'=F^{\Z_{\le 0}}$.
Note that such a group is residually finite if and only if $F$ is abelian \cite{Gru}.
\end{example}

\begin{example}
Let $q$ be a prime power, and set $\K=\F_q(\!(t)\!)$ and $\A=\F_q[t^{\pm 1}]$. Let $H$ be the Heisenberg group, consisting of upper-triangular square matrices of size three with 1's on the diagonal. Consider the action of $\Z$ on $H(\K)$ and $H(\A)$ defined by the automorphism
$$ \Phi:\begin{pmatrix}
  1 & x & z \\
  0 & 1 & y \\
  0 & 0 & 1 \\
\end{pmatrix}\mapsto \begin{pmatrix}
  1 & tx & t^2z \\
  0 & 1 & ty \\
  0 & 0 & 1 \\
\end{pmatrix}.$$
Then the group $H(\A)\rtimes_\Phi\Z$ is a cross-wired lamplighter. In this construction, $H$ can more generally be replaced by any nilpotent algebraic group with a contraction.
\end{example}

In the two previous examples, we have a natural embedding of $\Gamma$ as a cocompact lattice in a group of the form
$$(G_1\times G_2)\rtimes_{(\alpha_1,\alpha_2^{-1})}\Z,$$
where $(G_i,\alpha_i)$ is a contraction group, i.e.\ $\alpha_i$ is a contracting automorphism of $G_i$ in the sense that $\alpha_i^k(g)\to 1$ when $k\to +\infty$ (pointwise, but this automatically implies the convergence to be uniform on compact subsets).
Results of Glockner and Willis on the structure
of contraction groups \cite{GlocknerWillis} might thus prove useful to prove general structural results about cross-wired lamplighters.
However, it is not clear in general if contraction groups are enough. In general, if $\Gamma$ is a cross-wired lamplighter, then the closure of its projection on $\Aut(T)$, where $T$ is one of the two trees, is of the form
$$G\rtimes_\Phi \Z,$$
where $\Phi$ contracts $G$ modulo a compact subgroup, i.e.\ there exists a compact subgroup $K$ of $G$ such that for every $g$, we have $\Phi^k(g)\to 1$ in $G/K$, when $k\to +\infty$. Typically, if $G\rtimes_\alpha\Z$ is the full stabilizer of an end in $T_n$, then $\alpha$ contracts $G$ modulo a (non-trivial) compact subgroup. However, we do not know if a cross-wired lamplighter can project densely on this group.

In Section \ref{ex}, we develop these examples and ask a few additional questions.

Here are some more questions about cross-wired lamplighter we leave open:


\begin{question}~
\begin{enumerate}\item Is there a cross-wired lamplighter that is not ``virtually symmetric"? Here we define a cross-wired lamplighter to be symmetric if it admits an automorphism $\alpha$ such that if $\pi$ is the projection to $\Z$ (which is unique up to sign) then $\alpha\circ\pi\circ\alpha^{-1}=-\pi$, and virtually symmetric if it admits a symmetric subgroup of finite index.
\item Given a cross-wired lamplighter, is the pair $\{L,L'\}$ unique up to commensurability and automorphisms? Precisely, suppose that an abstract group $\Gamma$ is given two embeddings as a cross-wired lamplighter, giving rise to pairs $(L_1,L'_1)$ and $(L_2,L'_2)$. Is there (up to swapping $L_2$ and $L'_2$) an automorphism $\beta$ of $\Gamma$ such that $\beta(L_1)$ is commensurable with $L'_1$ and $\beta(L'_2)$ with $L'_2$?
\item Let $\Gamma$ be a closed cocompact subgroup of
$\Isom(\DL(n,n))$ {\em not} contained in $\Isom^+(\DL(n,n))$; it is easy to deduce from Theorem \ref{mainthm} that there
is a unique open normal subgroup $H$ such that $\Gamma/H$ is isomorphic to the
infinite dihedral group $D_\infty$. Is the extension $1\to H\to \Gamma\to D_\infty\to 1$
necessarily split?

\end{enumerate}
\end{question}






In Section \ref{idl}, we recall basic definitions
concerning Diestel-Leader graphs and prove Theorem
\ref{mainthm}.  In Section \ref{ex}, we describe various examples
of cross-wired lamplighters and we explain the choice
of title for this paper and name for these groups.


\medskip

\noindent\textbf{Acknowledgements.} We thank Laurent Bartholdi and Tullia Dymarz for useful remarks and corrections.

\section{Isometries of Diestel-Leader graphs}\label{idl}

\subsection{Preliminary remarks}
We begin by describing Diestel-Leader graphs. More detailed introductions to Diestel-Leader graphs,
lamplighter groups, and related concepts are given by Woess \cite{Woess}
and Wortman \cite{Wortman}.

Given an integer $n > 0$ we denote by $T_n$ the regular tree of degree $n+1$. We fix a preferred end $-\infty$ of $T_n$. Given two vertices $x, y \in V(T_n)$, the rays from $x$ and $y$ to $-\infty$ intersect to give a ray from some other vertex to $-\infty$. We shall denote this vertex by $\widehat{xy}$ and call it the confluent of  $x$ and $y$. By making a choice of preferred vertex $o$, we can define a Busemann function $b: V(T_n) \rightarrow \mathbf{Z}$ on the vertices of $T_n$ by
\[b(x) := d(x,\widehat{xo}) - d(o,\widehat{xo}).\]
This partitions the vertices of $T_n$ into level sets of $h$ which are called horocycles and denoted $\mathcal{F}_k$. A vertex in the horocycle $\mathcal{F}_k$ has one neighbour, its parent, in $\mathcal{F}_{k-1}$ and $n$ neighbours, its children, in $\mathcal{F}_{k+1}$.

Given two regular trees $T_m$ and $T_n$ and some choice of height functions $b$ and $b'$ respectively, the Diestel-Leader graph $\DL(m,n)$ is given by
\[V(\DL(m,n)) := \{(x,y) \in T_m \times T_n : b(x) + b'(y) = 0\},\]
where two vertices $(x,y)$ and $(x',y')$ are joined be an edge if $(x,x')$ is an edge in $T_m$ and $(y,y')$ is an edge in $T_n$.  Note that this description gives an embedding of $\DL(m,n)$ into $T_m \times T_n$. Any other choice of Busemann function on the trees gives rise to another such embedding (because any two Busemann functions on a regular tree are conjugate under some automorphism of the tree). In particular, each level set of the function $b(x) + b'(y)$ on $T_m \times T_n$ with edges defined analogously will be an embedded $\DL(m,n)$.

Note that we can define a Busemann function on $\DL(n,m)$ just by
taking $h(x,y)=b(x)=-b'(y)$.  We call the level sets of this Busemann function
a horosphere and denote it by $\mathcal H _k$.  A geodesic is called {\em vertical} if it
intersects each $\mathcal H_k$ exactly once.  Note that a horosphere is a product of horocycles,
one in $T_m$ one in $T_n$.

As observed in \cite{BNW}, the Diestel-Leader graphs only have unimodular isometry groups if $n=m$ and so only contain lattices
in this case.  

Let $U_0$ denote the isometries of $T_n$ which fix $-\infty$ and preserve the level sets of $h$.
Let $U'_0$ be the corresponding subgroup of isometries of $T_m$. It is clear that $U=U'_0 \times U_0$ is contained in $\Isom(\DL(m,n)$. The groups $U_0, U'_0$ and $U$ are elliptic, i.e.\ every compact subset is contained in an open compact subgroup.

Recall that
each $T_n$ (or $T_m$) has a preferred end which we labeled $-\infty$. Let $\phi_n$ be a hyperbolic isometry of $T_n$, of translation length $1$, and having $-\infty$ as repelling end and define $\phi_m$ similarly for $T_m$. Define the isometry $(x,y)\mapsto (\phi_m(x),\phi_n^{-1}(y))$ of $T_m\times T_n$. It restricts to an isometry of $\DL(n,n)$.  When $m=n$, there is also an isometry of $T_n \times T_n$ given
by interchanging the trees. This isometry restricts to an isometry of $\DL(n,n)$ which we call $\psi$.  We sometimes
refer to $\psi$ (or its conjugates) as a {\em flip}.

Thus $\Isom(\DL(n,n))$ contains the subgroup generated $U$, $\psi$ and $\phi$. The following result of Bartholdi, Neuhauser, and Woess \cite{BNW} shows that there are no further isometries.



\begin{prop}
\label{isometries}
For $m>n\ge 2$,
\[\Isom(\DL(m,n)) = \left(U \rtimes_{\phi} \mathbf{Z}\right)\]
For any $n\ge 2$,
\[\Isom(\DL(n,n)) = \left(U \rtimes_{\phi} \mathbf{Z}\right) \rtimes_{\psi} \mathbf{Z}/2\mathbf{Z}.\]
\end{prop}

Using the proposition, we fix some notation. We let $G=\Isom^+(\DL(n,n))$ be the index two subgroup $U \rtimes_{\phi} \mathbf{Z}$ when $m=n$ and $G=\Isom(\DL(m,n))$ when $m \neq n$. We let $\pi:G \rightarrow \Z$ be the natural projection so $U = \ker \pi$.


We also deduce some corollaries from the proposition.  The first is obvious.

\begin{cor}
\label{corollary:pairsoftrees} The full group $\Isom(\DL(m,n))$ acts
isometrically on the product $T_m \times T_n$ permuting the level
sets of the function $b(x)+b'(y)$.
\end{cor}

In particular, since every level set $b(x)+b'(y)=k$ is another copy
of the Diestel-Leader graph $\DL(m,n)$, which we will label for
clarity $\DL(m,n)_k$, any group $\Gamma < \Isom(\DL(m,n))$ has
infinitely many actions on $\DL(m,n)$. It will be important to us
that these actions are usually actually different. Let $\Gamma' = G \cap \Gamma$.

\begin{cor}
\label{cor:translationelement} If $\Gamma < \Isom(\DL(m,n))$ is a
closed, cocompact subgroup, then there is a natural number $m$ and an element $t$ in $\Gamma$
which is conjugate to $\phi^m$ such that $\Gamma' = (\Gamma' \cap U) \ltimes \langle t\rangle$.
\end{cor}

\begin{proof}
  It is clear that the image of $\pi$ restricted to $\Gamma'$ must be infinite for $\Gamma$ to be cocompact. Let $d$ be the (positive) generator of $\pi(\Gamma')$ and choose a lift $t$ of $d$
to $\Gamma$.  It is straightforward to check that $t$ has the desired properties.
\end{proof}

As promised in the introduction, we now give a direct proof that any group acting properly
and cocompactly on $\DL(n,n)$ admits a proper, transitive action on some $\DL(n',n')$.

\begin{prop}
\label{proposition:transitive} Let the locally compact group $\Gamma$ act properly and cocompactly on $\DL(n,n)$.
Then there is a positive integer $d$ such that $\Gamma$ acts
properly and transitively on $\DL(m^d,n^d)$.
\end{prop}

\begin{proof}
 Given  $\Gamma < \Isom(\DL(m,n))$ there is an element $t \in
\Gamma$ as in Corollary \ref{cor:translationelement}.  Since every element
of $\Gamma$ that does not fix $\mathcal H _0$ translates it to some $\mathcal H _{kd}$
with $k$ an integer,  it follows that the action of $\Gamma$ on
$\DL(m,n)$ can be replaced by an action of $\Gamma$ on a "collapsed"
Diestel-Leader graph $\DL(m^d, n^d)$.  We obtain the collapsed
Diestel-Leader graph from the first one by simply ignoring all
vertices that occur in level sets of $h$ that are not multiples of
$d$ and considering the edges to be vertical geodesic segment between
vertices in $\mathcal H_{dk}$ and vertices in $\mathcal H_{(d+1)k}$.

It is now clear that $\Gamma$ acts transitively on the set of horospheres of
$\DL(m^d, n^d)$.  To see that the action is transitive, we need to
see that the action along some horosphere is transitive.   Let the two trees be $T^1$
and $T^2$.  We let
$\Gamma'_0=\Gamma' \cup U$ the subgroup of elements stabilizing each horocycle.

We consider the horocycle
in $\mathcal H_0$ at height $0$ and note that $\Gamma'_0$ is cocompact on it. Choose
orbit representatives $(x_1,y_1), ...,
(x_l,y_l)$ for this action.   It is clear that there is some height $k^1$ (resp. $k^2$) such that all the
$x_j$ (resp. $y_j$) are on vertical geodesics going down from a single vertex $\bar x$ at height $k^1$ in $T^1$
(resp $\bar y$ at height $k^2$ in $T^2$).

Letting $k=|k^1 +  k^2|$ this implies that the $\Gamma_0'$ action on the horocycle containing $(\bar x, \bar y)$
in $\DL(m^d, n^d)_k$ is transitive.  If it is not, there is some other orbit, with representative $(x', y')$.
Let $x''$ (resp. $y''$) be any vertex at height zero below $x'$ (resp. $y'$) in $T^1$ (resp. $T^2$).  Since
$\Gamma_0'$ cannot move $(x',y')$ to $(\bar x, \bar y)$, it cannot move $(x'', y'')$ to any of our orbit
representatives for the $\Gamma_0'$ action on $\mathcal H_0$, a contradiction. This shows that the $\Gamma$ action on
$\DL(m^d, n^d)_k$ is transitive.
\end{proof}

\subsection{Proof of the theorem}


\begin{proof}
Let us prove (\ref{tta}). Let $\Gamma$ be a closed, cocompact subgroup of $G$. By Corollary \ref{cor:translationelement} there is $t\in\Gamma$ mapped to a generator of the cyclic group $\pi(\Gamma)$, such that $\Gamma=(U\cap\Gamma)\rtimes\langle t\rangle$.

For convenience we denote the two trees as $T$ and $T'$. Note that $t$ acts as a hyperbolic element on both $T$ and $T'$. Fix vertices $v\in T$ and $v'\in T'$ in the translation axes of $t$ in $T$ and $T'$. Let $L\subset\Gamma$ be the stabilizer of $v$ and $L'$ the stabilizer of $v'$ (for the action on $T$ and $T'$ respectively). Then $L$ and $L'$ are open in $\Gamma$ and contained in $U$. Since $v$ belongs to the axis of $t$, we have $tLt^{-1}\subset L$ and $t^{-1}L't\subset L'$. The open subgroup $L\cap L'$ is the stabilizer of $(v,v')$ and therefore is compact.
If $h\in\Gamma_0=\Gamma\cap\Ker(\pi)$, then $t^nht^{-n}\subset L$ for $n$ large enough. Indeed, write $h=(g,g')$. By assumption, $\pi(h)=0$, so $g$ has a fixed point $w$ in $T$. So $g$, hence $h$ also fixes the projection $w_1$ of $w$ on the axis of $v$, since $w_1$ is in the geodesic ray joining $w$ to the end at $-\infty$. So $h$ is contained in the conjugate of $L$ by some power of $t$, that is, $\bigcup_{n\in\Z}t^nLt^{-n}=\Gamma_0$. Similarly $\bigcup_{n\in\Z}t^{-n}L't^{n}=\Gamma_0$.

Let $S$ be the horosphere $\{x\in T:b(x)=0\}$ and similarly define $S'$ in $T'$. We see that two elements of $S\times S'$ are contained in the same orbit under $\Gamma$ if and only if they are contained in the same orbit under $\Gamma_0$. In particular, $\Gamma_0$ has finitely many orbits for its action on $\Gamma_0v\times\Gamma_0v'$, or equivalently for its action on $\Gamma_0/L\times\Gamma_0/L'$. This exactly means that $L\backslash\Gamma_0/L'$ is finite.

Therefore there exists a finite subset $F$ of $\Gamma_0$ such that $\Gamma_0=LFL'$. For some $k\le 0$ we have $F\subset t^kLt^{-k}$.

If we had chosen $t^kv$ instead of $v$, the open subgroup $L$ would be replaced by $t^kLt^{-k}$, and we would get $\Gamma_0=LL'$ without altering the other conclusions.

Let us prove (\ref{ttb}). Given the data of (\ref{ttb}), and denoting by $d$ the cardinality of $L\backslash H/L'$, we give a construction giving the following
\begin{itemize}\item $\Gamma$ has a proper, cocompact action on $\DL(m,n)$ with exactly $d$ orbits;
\item $\Gamma$ has a proper, transitive action on $\DL(n^k,n^k)$ for some $k\le d$.
\end{itemize}

We see that the group $\Gamma$ is naturally identified with the ascending HNN extension $\HNN(H,L,t)$. So it has a natural action on its Bass-Serre tree $T$. Recall the definition of this tree. Write $L_n=t^nLt^{-n}$. Its vertex set can be written as the disjoint union $\bigsqcup_{n\in\Z} H/L_n$, and the vertices are of the form $\{aL_n,aL_{n+1}\}$; in particular, this tree is regular of degree $1+[L:tLt^{-1}]$; the action of $H$ on each coset $H/L_n$ is the natural one, and the action of $t$ is given by $t\cdot (aL_n)=(tat^{-1})L_{n+1}$. Since $tL_nt^{-1}=L_{n+1}$, this is a well-defined continuous action. The stabilizer of a vertex is $L$. If we map $T$ to $\Z$ by sending each vertex in $H/L_n$ to $n$, we obtain a Busemann function $b$.

We can carry out the same construction with $L'$ and get a natural action on a tree $T'$, regular of degree $1+[L':t^{-1}L't]$, with a Busemann function. The diagonal action on $T\times T'$ has stabilizer $L\cap L'$, so is proper. It preserves the Diestel-Leader graph $D$ of equation $b(x)+b'(y)=0$.

We want to check that this action has exactly $d$ orbits. Using the action of $t$, every orbit intersects the subset of equation $b(x)=b(y)=0$. So we are reduced to proving that the action of $H$ on $H/L\times H/L'$ has $d$ orbits, but these orbits exactly correspond to elements of the double coset space $L\backslash H/L'$.

This ends the construction. Now, as in the proof of (\ref{tta}), if we replace $L$ by its conjugate by some power of $t$, we find another subgroup
$L_1$, so that all hypotheses of (\ref{ttb}) are fulfilled with in addition $L_1L'=H$, so the construction above provides a transitive action.

In both (\ref{tta}) and (\ref{ttb}) the fact that $\Gamma$ is discrete implies that it is finitely generated (because it acts properly cocompactly on a locally finite graph) and that $m=n$ because otherwise $\Isom(\DL(m,n))$ is not unimodular and hence does not admit lattices.
\end{proof}

\section{Examples and further comments}\label{ex}

\subsection{Examples}
We now provide some more examples of cross-wired lamplighters which are
not lamplighters, i.e.\ which are not of the form $F \wr {\mathbf Z}$ for some
finite group $F$.  We provide a specific example of a general construction and
indicate the general construction.

We let $\N$ be the Heisenberg group over the local field $\F_p(\!(t)\!)$
and let $N$ be the subgroup matrices with values in  $\F_p[t,
t{\inv}]$.  We have two natural embeddings of $N$ in $\N$, one which
takes $t$ to $t$ and one which takes $t$ to $t{\inv}$.  Combining
these two embeddings, we obtain a map $i:N \rightarrow \N \times
\N$. It is standard to note that $\N$ has a contracting automorphism
$\alpha$, this acts as multiplication by $t$ on the $x$ and $y$
coordinates and by multiplication by $t^2$ on the $z$ coordinate. We
define the group $N \rtimes \Z$ where the $\Z$ action is defined by
restricting $\alpha \times \alpha$ on $\N \times \N$ to $i(N)$. To
verify the statement of Theorem \ref{mainthm}, we let $L=i(\F_p[t])$
and $L'=i(\F_p[t{\inv}])$.  It is straightforward to verify that
the conditions of the theorem are satisfied.

One can easily adapt this construction to other nilpotent
groups $\N$ over local fields of positive characteristic which
admit contracting automorphisms.  Note that in any construction of
this kind, we are constructing a linear group, which is therefore
residually finite.

Note that when $\Gamma = F \wr \mathbf{Z}$ then $\Gamma$ is residually finite
if and only if $F$ is abelian. This is easily checked, since in any finite
quotient of $\Gamma$, the image of $F_i$ is the same as the image of $F_{i+p}$
for some finite $p>1$. Since $F_{i+p}$ and $F_{i}$ commute, it follows that the image of $F_i$
must be abelian.

It would be interesting to construct examples of cross-wired lamplighters which
are not commensurable to either wreath products or linear groups.  Determining whether
or not such examples exist is necessary to answer the questions asked at the end of the
introduction.  The reader should note that groups commensurable to wreath products
are not necessarily wreath products themselves.

\begin{remark}Another generalization of lamplighter groups $F\wr\mathbf{Z}$ ($F$ finite abelian) are Cayley machines associated to arbitrary finite groups \cite{SiSt}. It is suggested to us by L.\ Bartholdi that these might provide other examples of cross-wired lamplighters.\end{remark}



\begin{remark}
If $G$ is a locally compact, compactly generated group and admits a cocompact, continuous, faithful action on a tree with no degree one vertex, then $G$ has no nontrivial compact normal subgroup. Indeed, every cocompact action on such a tree has to be minimal. It follows that every closed cocompact subgroup in $\Isom(\DL(m,n))$ has no nontrivial compact normal subgroup. Thus a locally compact group $\Gamma$ admitting a proper cocompact action on $\DL(m,n)$ ($m,n\ge 2$) has a unique maximal compact normal subgroup $W(\Gamma)$, and $\Gamma$ is actually isomorphic to a closed cocompact subgroup of $\DL(m,n)$ if and only if $W(\Gamma)=\{1\}$.

Here is an example of a discrete group $\Gamma$ with a proper cocompact action on some $\DL(n,n)$ such that $W(\Gamma_1)$ is nontrivial for every finite index subgroup $\Gamma_1$ of $\Gamma$. Namely, let $F$ be a finite group and $Z$ a nontrivial central subgroup of $F$ contained in the derived subgroup $[F,F]$. Start from the wreath product $F\wr\mathbf{Z}=F^{(\mathbf{Z})}\rtimes\Z$ and mod out by the subgroup $Z_0$ of $Z^{(\mathbf{Z})}$ of families $(z_n)_{n\in\Z}$ such that $\sum_{n\in\Z}z_n=0$, and $\Gamma=(F\wr\mathbf{Z})/Z_0$ ($Z_0$ is normal in $F\wr\mathbf{Z}$ because $Z$ is central in $F$). The image of $Z^{(\mathbf{Z})}$ in $\Gamma$ is a central subgroup isomorphic to $Z$, let us call it $Z$. By Gruenberg \cite{Gru}, every finite index subgroup of $F\wr\mathbf{Z}$ contains $[F,F]^{(\mathbf{Z})}$. Thus every finite index subgroup $\Gamma_1$ of $\Gamma$ contains the finite normal subgroup $Z$ (and thus it easily follows that $W(\Gamma_1)=Z$).
\end{remark}

\subsection{The name ``Cross-wired lamplighters".}

A lamplighter group $ \sum F \rtimes \Z$ is usually interpreted as follows.  The generator of $\Z$ describes
the walk of a lamplighter along an infinite row of lamps with $|F|$ states.  The generators of $F$ represent
the lamplighters ability to change the state of the lamp at his current position.  (The description is most
intuitive where $F=\Z / 2\Z$ and lamps are simply on or off.)

All of our examples of cross-wired lamplighters have a similar but more complex structure.  The generator of $\Z$ can again be
viewed as letting the lamplighter walk along an infinite row of lamps.  And in each example there is a finite group $F$ which can
again be viewed as allowing the lamplighter to change the state of some lamp at his current position.  However,
in this setting, the wires of the lamps are ``crossed" and changing the setting of the lamp at the current position
may change the setting of other lamps at other positions.  In all of our examples these ``cross-wires" arise from
a kind of non-commutativity based on a nilpotent structure.  It would be interesting to know if all ``cross-wired lamplighters" exhibit similar structure.

\bibliographystyle{alpha}

\begin{thebibliography}{10}
\bibitem{BNW}
L. Bartholdi, M. Neuhauser, W. Woess. \emph{ Horocyclic products of
trees}. J. European Math. Society 10 (2008) 771--816.

\bibitem{DiestelLeader}
R. Diestel, I. Leader.
\emph{ A conjecture concerning a limit of non-Cayley graphs}.
J. Algebraic Combin. 14 (2001), no. 1, 17--25.

\bibitem{Dyubina} A. Dyubina. \newblock {\em Instability
the virtual solvability and the property of being virtually
torsion-free for quasi-isometric groups}. \newblock Int. Math.
Res. Not. {\bf 21}, 1097--1101, 2000.

\bibitem{EFW1}
A. Eskin, D. Fisher, K. Whyte. \emph{Quasi-isometries and rigidity
of Solvable Groups}. Pure and Applied Mathematics Quarterly 3 (2007)
927--947.

\bibitem{EFW2}
A. Eskin, D. Fisher, K. Whyte. \emph{ Coarse Differentiation of
Quasi-isometries I: spaces not quasi-isometric to Cayley graphs}. Annals of Math. 176 (2012) 221--260.


\bibitem{EFW3}
A. Eskin, D. Fisher, K. Whyte. \emph{Coarse Differentiation of
Quasi-isometries II: rigidity for Sol and Lamplighter groups}. ArXiv 0706.0940v1.

\bibitem{GlocknerWillis}
H. Glockner, G. Willis. \emph{Classification of the simple factors appearing
in the composition series of totally disconnected contraction groups},  J. f\"ur Reine Angew. Math. (Crelles Journal) 643 (2010) 141--169.

\bibitem{Golovin}
O. Golovin. \emph{Nilpotent products of groups.} Amer. Math. Soc.
Transl. (2) 2 (1956), 89--115.

\bibitem{Gru} K. W. Gruenberg.
Residual properties of infinite soluble groups.  Proc.
London Math. Soc. (3) 7, 1957, 29--62.

\bibitem{MacHenry}
T. MacHenry. \emph{The tensor product and the 2nd nilpotent product
of groups.} Math. Z.  73  1960 134--145.

\bibitem{SiSt} P. V. Silva and B. Steinberg. \emph{On a class of automata groups generalizing lamplighter groups}. Internat. J. Algebra Comput. 15 (2005), 1213--1235.

\bibitem{Woess}
W. Woess. \emph{Lamplighters, Diestel-Leader graphs, random walks,
and harmonic functions}. Combinatorics, Probability \& Computing 14
(2005) 415--433.

\bibitem{Wortman}
K. Wortman. \emph{ A Finitely-presented solvable group with a small
quasi-isometry group}. Michigan Math. J. 55 (2007), 3--24.
\end{thebibliography}

\end{document}